\documentclass[11pt]{article}
\usepackage{amsmath,amsthm,amssymb,hyperref}
\newcommand{\remove}[1]{}
\newtheorem{thm}{Theorem}[section]
\newtheorem{claim}[thm]{Claim}
\newtheorem{lem}[thm]{Lemma}
\newtheorem{define}[thm]{Definition}
\newtheorem{cor}[thm]{Corollary}

\def\L{{\mathcal L}}
\def\SC{{\mathbf{SC}}}
\def\fl{\text{fl}}
\def\SG{\text{SG}}
\def\F{{\mathbb{F}}}

\def\R{{\mathbb{R}}}

\def\C{{\mathbb{C}}}

\newcommand{\ip}[2]{\langle #1,#2 \rangle}

\def\_{\,\,\,\,\,}

\def\supp{\textsf{supp}}

\def\rank{\textsf{rank}}

\def\dimension{\textrm{dim}}
\def\affinedim{\textrm{affine-dim}}

\def\fplus{{\,+_f\,}}

\newcommand{\eps}{\epsilon}

\begin{document}

\title{Improved rank bounds for design matrices and a new proof of Kelly's theorem}

\author{Zeev Dvir\thanks{Department of Computer Science and Department of Mathematics, Princeton University.
Email: \texttt{zeev.dvir@gmail.com}. Research  supported by NSF grants CCF-1217416 and CCF-0832797.} \and
Shubhangi Saraf\thanks{Department of Computer Science and Department of Mathematics, Rutgers University.
Email: \texttt{shubhangi.saraf@gmail.com}.} \and
Avi Wigderson\thanks{School of Mathematics, Institute for Advanced Study.
Email: \texttt{avi@ias.edu}.}}

\date{}
\maketitle



\begin{abstract}
We  study the rank of complex sparse matrices  in which the supports of different columns have small intersections. The rank of these matrices, called {\em design} matrices, was the focus of a recent work by Barak et. al. \cite{BDWY11} in which they were used to answer questions regarding point configurations. In this work we derive near-optimal rank bounds for these matrices and use them to obtain asymptotically tight bounds in many of the geometric applications. As a consequence of our improved analysis, we also obtain a new, linear algebraic, proof of Kelly's theorem, which is the complex analog of the Sylvester-Gallai theorem. 
\end{abstract}

\section{Introduction}

The classical Sylvester-Gallai (SG) Theorem  states the following: \emph{Given any finite set of points in the Euclidean plane, not all on the same line, there exists a line passing through exactly two of the points.} This result was first conjectured by Sylvester in 1893~\cite{Sylvester} and then proved independently by Melchior in 1940~\cite{Melchior} and Gallai in 1943 (in an answer to the same question independently posed by Erdos~\cite{Erdos43}). The following complex variant of the Sylvester-Gallai Thoerem was proved by Kelly~\cite{Kel86} in response to a question of Serre: \emph{Given any finite set of points in $\C^d$, not all on the same complex two-dimensional plane, there exists a line passing through exactly two of the points.} The above result is tight over the complex numbers, since there are two-dimensional configurations of points satisfying the condition on triples.  The survey by Borwein and Moser~\cite{BM90} gives a very good overview of the SG Theorem, its proofs and the different variants that have been studied in the past.  One application of our techniques, discussed later, is a new proof of Kelly's theorem, which is significantly simpler than Kelly's original proof and is very different than the recent elementary proof of \cite{EPS06}.
  
In recent years, variants of the SG theorem have been useful in studying certain structural questions arising in  theoretical computer science. Variants of the SG theorem were useful in understanding the structure of low-depth arithmetic circuits \cite{DvirShpilka06, KS09, SS10}. Quantitative versions of the SG theorem were shown to be closely linked to the structure of linear Locally Correctable Codes~\cite{BDWY11, BDSS11}. These applications join a growing number of papers in which geometric theorems regarding point/line arrangements are finding applications in theoretical computer science. We refer the reader to the recent survey \cite{Dvir-survey} for an overview of some of these applications.

\subsection{Rank of design matrices}

Motivated by the application to point configurations, \cite{BDWY11} studied the rank of certain matrices, called {\em design} matrices, and proved lower bounds on their rank. These bounds were then used to prove quantitative analogs of the SG theorem. We begin by defining design matrices formally. For a vector $v \in \F^n$, where $\F$ is a field, we denote by $\supp(v) = \{i \in [n]\,|\, v_i \neq 0\}.$

\begin{define}[Design matrix]\label{def-designmatrix}
Let $A$ be an $m \times n$ matrix over a field $\F$. Let $R_1,\ldots,R_m \in \F^n$ be  the rows of $A$ and let $C_1,\ldots,C_n \in \F^m$ be the columns of $A$. We say that $A$ is a {\em $(q,k,t)$-design matrix} if
\begin{enumerate}
\item For all $i \in [m]$, $|\supp(R_i)| \leq q$.
\item For all $j \in [n]$, $|\supp(C_j)| \geq k$.
\item For all $j_1 \neq j_2 \in [n]$, $|\supp(C_{j_1}) \cap \supp(C_{j_2}) | \leq t$.
\end{enumerate}
\end{define}

A `typical' setting of the parameters that often arises in applications is when $q$ is a small constant, $m \approx n^2$, $k \approx n$ and $t$ is  a constant. The main result in \cite{BDWY11} is the following rank bound.

\begin{thm}[\cite{BDWY11}]\label{thm-bdwymain}
Let $A$ by an $m \times n$  matrix. If $A$ is a $(q,k,t)$ design matrix then $$\rank(A) \geq n - \left(\frac{qtn}{2k}\right)^2$$
\end{thm}

For the aforementioned `typical' setting of the parameters, one get a lower bound of $n - O((n/k)^2)$ on the rank. By improving a key lemma in   \cite{BDWY11} using a more careful analysis, we are able to prove the following new bound.

\begin{thm}\label{thm:improvedrankbound}
Let $A$ by an $m \times n$  matrix. If $A$ is a $(q,k,t)$ design matrix then $$\rank(A) \geq \frac{n}{1 + \frac{q(q-1)mt}{nk^2}} \geq n - \frac{mtq(q-1)}{k^2}$$
\end{thm}

As a corollary, we get:
\begin{cor}\label{cor:improvedwithoutm}
Let $A$ by an $m \times n$  matrix. If $A$ is a $(q,k,t)$ design matrix then $$\rank(A) \geq \frac{n}{1 + \frac{q(q-1)t}{k}} \geq  n - \frac{ntq(q-1)}{k}.$$
\end{cor}
\begin{proof}
If $m \leq nk$ then we are done (substitute $m=nk$ into the bound in Theorem~\ref{thm:improvedrankbound}). Otherwise,   we can remove rows of $A$ until we are left with a new matrix $A'$ that has exactly $m' = nk$ rows and s.t $A'$ is also a $(q,k,t)$-design matrix (as long as $m > nk$ there has to be a row we can remove and maintain that each column has at least $k$ non zeros). Now, apply Theorem~\ref{thm:improvedrankbound} on $A'$ and use the fact that $\rank(A) \geq \rank(A')$.
\end{proof}

Here, for the `typical' setting, we get a rank bound of $n - O(n/k)$, which is asymptotically better then the one obtained in Theorem~\ref{thm-bdwymain}.

We also prove a variant of Theorem~\ref{thm:improvedrankbound} in which $q$ does not appear. Note that if each column of $A$ has support of size exactly $k$, then the total number of nonzero entries in $A$ is $nk$. Thus the average size of the support of a row would be $nk/m$. In general it can be shown that replacing $q$ with $nk/m$ in Theorem~\ref{thm:improvedrankbound} would give a false statement. However, we show that we can replace $q$ with $O(nt/k)$. This is exactly the average row-support when $m = O(k^2/t)$  and, in this regime, the bound on the rank (which is now independent of $q$) is tight\footnote{Observe  that for an $m \times n$ matrix that is a $(q,k,t)$ design, if $n= \Omega( k/t)$, then $m = \Omega(k^2/t)$. This follows by a simple inclusion-exclusion argument on the union of the supports of any $\Omega(k/t)$ columns of $A$.}. 

\begin{thm}\label{thm:averageq}
Let $A$ by an $m \times n$  matrix. If $A$ is a $(q,k,t)$ design matrix with $k \leq nt$ then $$\rank(A) \geq n - \frac{6mn^2t^3}{k^4}. $$
\end{thm}

As before, by replacing $m$ with $nk$ we get the following corollary.

\begin{cor}\label{cor:avgqwithoutm}
Let $A$ by an $m \times n$  matrix. If $A$ is a $(q,k,t)$ design matrix with $k \leq nt$ then $$\rank(A) \geq n - \frac{6n^3t^3}{k^3}. $$
\end{cor}

\paragraph{Related work:} The problem of bounding the rank of matrices with certain patterns of zeros and non-zeros is not new, and  has been  studied in the past in a variety of contexts. One line of research comes from Hamada's conjecture~\cite{ham73, jt09,biw07}. In this setting however, the notion of being a design is stricter than the notion we use in this paper.  Another line of research that studies the zero-nonzero patterns of matrices and their ranks has many applications to graph theory\cite{fh07}. Rank bounds on matrices with `sign patterns' of positive and negative entries have also received a great deal of attention in recent years~\cite{RS08,for02}, and it would be interesting to see if our techniques can say anything meaningful in this setting.

\subsection{\bf Square design matrices and monotone rigidity}
Theorem~\ref{thm:averageq}, which removes the dependence on $q$, allows us to get meaningful lower bounds on the rank
of {\it square} design matrices. The results of \cite{BDWY11} did not give anything for such matrices. Let $A$ be an $n\times n$ matrix such that every column has support of size $k \approx {\sqrt n}$ and such that
for every two columns, the size of the intersection of the supports of the two columns  $t$ is $O(1)$. For instance, the zero-nonzero
pattern of the projective plane incidence matrix has this structure.  In this case we can obtain a bound of $\Omega(n)$ on the rank of such a matrix - not by applying our rank bound directly
but by applying Theorem~\ref{thm:averageq} to the matrix after deleting a linear number of columns.  

A simple consequence of this result, proved in Section~\ref{sec:sqrb}, is that, if one takes the $n \times n$ incidence matrix of the projective plane and changes a small number of $1$'s in the matrix to arbitrary values, then the resulting matrix has high (linear in $n$) rank. This can be viewed as a restricted model of the {\em matrix rigidity} problem of Valiant \cite{Val77}. A matrix is rigid if changing  a small number of its positions cannot decrease its rank by much. Valiant showed that a linear circuit computing a transformation given by a rigid matrix cannot have linear size and logarithmic depth. Hence, the problem of finding an explicit rigid matrix will imply circuit lower bounds that are beyond our reach at this point. Our restricted model allows one to only change positions in the matrix that are non-zero. Even though this result does not yield any interesting result on circuit lower bounds, we find it encouraging in that it gives a way to control rank under {\em some} type of perturbations. The full details are given in Section~\ref{sec:sqrb}.


\subsection{Configurations with many collinear triples}

 Given a set of points $v_1,\ldots,v_n \in \C^d$, we call a line that passes through exactly two of the points of the set an {\it ordinary} line. A line passing through at least three points is called {\em special}. We will use $\dimension(v_1,\ldots,v_n)$ to denote the dimension of the linear span of $v_1,\ldots,v_n$ and by $\affinedim(v_1,\ldots,v_n)$ the dimension of the affine span of $v_1,\ldots,v_n$ (i.e., the minimum $r$ such that $v_1,\ldots,v_n$ are contained in a shift of a linear subspace of dimension $r$).

The main geometric application studied in \cite{BDWY11} was to extend the SG theorem to configurations of points termed $\delta$-SG configurations.

\begin{define}[$\delta$-SG configuration]
We say that a set of points  $v_1, v_2, \ldots, v_n \in \C^d$ is a $\delta$-SG configuration  if for every $v_i$, $i\in [n]$, at least $\delta (n-1)$ of the remaining points lie on special lines through $v_i$. 
\end{define}

Setting $\delta=1$ we can state the original SG theorem  as saying that any $1$-SG configuration $v_1,\ldots,v_n$ (over $\R$) has $\affinedim(v_1,\ldots,v_n) \leq 1$ (i.e., is contained in a line). Kelly's theorem \cite{Kel86} obtains the bound $\affinedim(v_1,\ldots,v_n) \leq 2$ for $1$-SG configurations over $\C$ (the bound $2$ is tight in this case). The following theorem, proved in \cite{BDWY11} gives a generalization of Kelly's theorem to $\delta$-SG configurations.

\begin{thm}[\cite{BDWY11}]\label{thm:delta-sg-bdwy}
Let $V = \{v_1,\ldots,v_n\} \subset \C^d$ be a $\delta$-SG configuration. Then $$\affinedim(v_1, \ldots, v_n) \leq 13/\delta^2.$$ When $\delta=1$ one gets a bound of $9$ on the affine dimension.
\end{thm}

There are two shortcomings of this theorem. The  first is the quadratic dependence on $\delta$. Placing the points on $1/\delta$ lines in general position one can construct a $\delta$-SG configuration with dimension $\Omega(1/\delta)$. It was left as an open question in \cite{BDWY11} to close this quadratic gap between the lower and upper bound on the dimension of $\delta$-SG configurations. The second issue is that one does not recover Kelly's theorem from the proof of \cite{BDWY11}, but only an inferior bound of $9$ on the dimension. We are able to correct both of these issues.

\begin{thm}\label{thm:delta-sg}
Let $V = \{v_1,\ldots,v_n\} \subset \C^d$ be a $\delta$-SG configuration. Then $$\affinedim(v_1, \ldots, v_n) \leq 12/\delta.$$ When $\delta=1$ one gets a bound of $2$ on the affine dimension.
\end{thm}

There are two known proofs of Kelly's theorem. Kelly's original proof, answering  a question by Serre, used deep results from algebraic geometry. An elementary proof was recently found by Elkies, Pretorius and Swanepoel~\cite{EPS06}. Our proof is conceptually very different from both of these and uses only elementary linear algebra.

Theorem~\ref{thm:delta-sg} is proved, as in \cite{BDWY11}, by reduction to the rank bound for design matrices. One constructs a design matrix whose co-rank bounds the dimension of the configuration and then applies one of the bounds on the rank of these matrices.

\paragraph{Average-case version:} A natural variant on the definition of a $\delta$-SG configuration is to only require the presence of many pairs of points on special lines (instead of requiring each point to belong to many such pairs). In \cite{BDWY11} it was shown that:

\begin{thm}[\cite{BDWY11}]\label{thm:SGav2}
Let $V = \{v_1,\ldots,v_n\} \subset \C^d$ be a set of $n$ points such that at least $\delta n^2$ (unordered) pairs of  them lie on special lines. Then there exists a subset $V' \subset V$ such that $|V'| \geq (\delta/6)n$ and so that $V'$ is a $\delta$-SG configuration.
\end{thm}

Combining this result with our improved bound on the dimension  of $\delta$-SG configurations we get the following improvement to a theorem from \cite{BDWY11}.

\begin{cor}\label{cor:SGav2}
Let $V = \{v_1,\ldots,v_n\} \subset \C^d$ be a set of $n$ distinct points. Suppose that there are at least $\delta n^2$ unordered pairs of points in $V$ that lie on a special line. Then there exists a subset $V' \subset V$ such that $|V'| \geq (\delta/6) n$ and $\affinedim(V') \leq O(1/\delta)$.
\end{cor}

\subsection{Flats of higher dimension} 

A \emph{$k$-flat} is an affine subspace of dimension $k$. Let $\fl(v_1,\ldots,v_k)$ denote the flat spanned by these $k$ points (it can have dimension at most $k-1$). We call $v_1,\ldots,v_k$ \emph{independent} if their flat is of  dimension $k-1$,
and say that $v_1,\ldots,v_k$ are \emph{dependent} otherwise. Considering some fixed finite subset $V \subset \C^d$ of size $n$ we call a $k$-flat \emph{ordinary}  if its intersection with $V$ is contained in the union of a $(k-1)$-flat and a single point (this agrees with the definition of an ordinary line when $k=1$). A $k$-flat is \emph{elementary} if its intersection with $V$ has exactly $k+1$ points. Notice that for $k=1$ (when flats are lines) the two notions of ordinary and elementary coincide.

The next  definition generalizes the notion of a $\delta$-SG configuration for higher dimensional flats in two different ways (using ordinary/elementary flats). For $k=1$ both definitions coincide.

\begin{define}[$\delta$-$\SG_k$, $\delta$-$\SG_k^*$]
The set $V$ is a $\delta$-$\SG_k^*$ configuration if for every independent $v_1,\ldots,v_k \in V$ there are at least $\delta n$ points $u \in V$ such that either $u \in \fl(v_1,\ldots,v_k)$ or the $k$-flat $\fl(v_1,\ldots,v_k,u)$ contains a point $w$ outside $\fl(v_1,\ldots,v_k) \cup \{ u\}$. The set $V$ is a $\delta$-$\SG_k$ configuration if for every independent $v_1,\ldots,v_k \in V$ there are at least $\delta n$ points $u \in V$ such that either $u \in \fl(v_1,\ldots,v_k)$ or the $k$-flat $\fl(v_1,\ldots,v_k,u)$ is not elementary. Notice that a $\delta$-$\SG_k^*$ configuration is also a $\delta$-$\SG_k$ configuration and that for $k=1$ both are the same.
\end{define}

In \cite{BDWY11}, the following theorem was proved. 
\begin{thm}[\cite{BDWY11}]\label{thm-highdim-bdwy}
Let $V$ and $V^*$ be a $\delta$-$\SG_k$ and a $\delta$-$\SG_k^*$ configurations respectively in $\C^d$. Then:
\begin{enumerate}
	\item $\affinedim(V^*) \leq O\left( (k/\delta)^{2} \right)$.
	\item $\affinedim(V) \leq 2^{C^k} / \delta^2$, where $C>0$ is a universal constant.
\end{enumerate}
\end{thm}

Prior to this result, the only known bound for configurations with many special $k$-flats was a result proved by Hansen and Bonnice-Edelstein \cite{Han65,BE67} which gives a bound of $O(k)$ on the dimension of a $1$-$\SG_k$ (or $1$-$\SG_k^*$) configuration over the reals. Since Theorem~\ref{thm-highdim-bdwy} is proved in a black-box manner using the result for $\delta$-$\SG$ configurations we can plug-in our improvement, given in Theorem~\ref{thm:delta-sg}, to obtain the following. 

\begin{thm}\label{thm-highdim-new}
Let $V$ and $V^*$ be a $\delta$-$\SG_k$ and a $\delta$-$\SG_k^*$ configurations respectively in $\C^d$. Then:
\begin{enumerate}
	\item $\affinedim(V^*) \leq O\left( k/\delta \right)$.
	\item $\affinedim(V) \leq C^k / \delta$, where $C>0$ is a universal constant.
\end{enumerate}
\end{thm}

Notice that, whereas the improvement for $\delta$-$\SG_k^*$ configurations is only quadratic, the improvement for $\delta-SG_k$ configurations is exponential (this is due to the way the basic bound is amplified in the induction on $k$). The proof of this theorem is identical to the proof of Theorem~\ref{thm-highdim-bdwy} appearing  in~\cite{BDWY11}, only with Theorem~\ref{thm:delta-sg-bdwy} replaced by Theorem~\ref{thm:delta-sg}. For completeness, we give the details in Section~\ref{sec:hdsg}.

\subsection{A variation on Freiman's Lemma}

Consider a finite set $A$ in some abelian group. One can define the {\em sumset} $A+A = \{a_1 + a_2 \,|\, a_1,a_2 \in A\}$ in a natural way. A well-known result in additive combinatorics is the following lemma, known as Freiman's lemma, which derives structural information on $A$, given bounds on the size of $A+A$.
\begin{lem}[Freiman's lemma. See \cite{TV06}]\label{lem-freiman}
Let $A$ be a finite subset of $\R^d$ and suppose $|A+A| \leq K|A|$. Then $A$ is contained in a linear subspace of dimension at most  $\lfloor K-1 \rfloor$.
\end{lem}
 
Clearly, the condition $|A+A| \leq K|A|$ can be replaced by $|\{(a_1 + a_2)/2 \,|\, a_1,a_2 \in A\}| \leq K|A|$, where we replace sums with mid-points. Surprisingly enough, the original proof of this lemma works also when we replace mid points with {\em any point on the line segment connecting the two points}. More formally, for two sets  $A,B \subset \R^d$ and any function $f :A \times B \mapsto \R^d$ we can define  $A \fplus B = \{ f(a,b) \,|\, a \in A,\, b \in B \}$. Then, as long as $f(a,b)$ is on the line segment connecting $a,b$ (and is different from $a,b$) we get the same conclusion as in Lemma~\ref{lem-freiman}, assuming  $|A \fplus B| \leq K|A|$. 

Intuitively, our results for $\delta$-$\SG$ configurations are of a similar flavor since the assumption of Freiman's Lemma (in its generalized form just stated) implies the existence of many pairs of points on special lines. We are able to use our techniques to derive the following theorem (whose proof appears in Section~\ref{sec:freiman}). 
 
\begin{thm}\label{thm-freiman-new}
Let $A$ be a finite subset of $\C^d$ and let $f : A \times A \mapsto \C^d$ be any function such that for all $a_1 \neq a_2 \in A$ we have $f(a_1,a_2) = \alpha a_1 + (1-\alpha) a_2$ for some $\alpha = \alpha(a_1,a_2) : A \times A \mapsto \in \C \setminus \{0,1\}$ (i.e., $f(a_1,a_2)$ is on the line passing through $a_1,a_2$ minus the two points $a_1,a_2$). Suppose that $$ |A \fplus A| \leq K\cdot A.$$ Then $\dim(A) \leq  O(K^{2})$	
\end{thm}

This theorem relaxes the conditions of Freiman's lemma by allowing (a) points in complex space and (b) the value of $f(a_1,a_2)$ to be outside the convex hull of $a_1,a_2$. On the other hand, we get a worse bound of $O(K^2)$ instead of $O(K)$. We do not know if this quadratic loss is needed or not.

\subsection{Organization} 
In Section~\ref{sec:prelim} we introduce some preliminaries related to the technique of matrix scaling. In Section~\ref{sec:rb} we prove Theorem~\ref{thm:improvedrankbound} and Theorem~\ref{thm:averageq}. The rank bound for square design matrices and the application for monotone linear circuits is given in Section~\ref{sec:sqrb}. In Section~\ref{sec:apps} we prove our main application, Theorem~\ref{thm:delta-sg}. In Section~\ref{sec:hdsg} we prove the high dimensional variant. In Section~\ref{sec:freiman} we prove Theorem~\ref{thm-freiman-new}. 

\section{Preliminaries -- Matrix Scaling}\label{sec:prelim}

One of the most important ingredients in the proof of the rank bound for design matrices is the notion of {\it matrix scaling}. Informally, the matrix scaling theorem states that if a matrix does not have any large zero sub-matrices, then one can multiply the rows and columns of the matrix by non-zero scalars so that all the row sums are equal and all the column sums are equal (assuming the entries are non-negative). 

The technique of matrix scaling originated in a paper of Sinkhorn \cite{Sinkhorn} and has been
widely studied since then (see \cite{LSW98} for more background). It was used in \cite{BDWY11}  for the first time to study design matrices, and we build upon their work and extend it.

We first set up some notation. For a complex matrix $X$, we let $X^\ast$ denote the matrix $X$ conjugated and transposed. Also we let $X_{ij}$ denote the $(i,j)$ entry of $X$. 
For two complex vectors $u, v \in \C^m$, we denote their inner product by $\ip{u}{v} = \sum_{i=1}^m u_i\cdot\overline v_j$ and let $\|u\| = \sqrt{\ip{u}{u}}$ denote the $\ell_2$ norm of the complex vector $u$.

\begin{define}\label{def-scaling}[Matrix scaling]
Let $A$ be an $m \times n$ complex matrix. Let $\rho \in \C^{m}, \gamma \in \C^n$ be two complex vectors
with all entries non-zero.
We denote by $$ \SC(A,\rho,\gamma)$$ the matrix obtained from $A$ by multiplying the $(i,j)$'th element of $A$ by $\rho_i \cdot \gamma_j$. We say that two matrices $A,B$ of the same dimensions are a scaling of each other if there exist non-zero vectors $\rho,\gamma$ such that $B = \SC(A,\rho,\gamma)$.
It is easy to check that this is an equivalence relation. We refer to the elements of the vector $\rho$ as the {\em row scaling coefficients} and to the elements of $\gamma$ as the {\em column scaling coefficients}. Notice that two matrices which are a scaling of each other have the same rank and the same pattern of zero and non-zero entries.
\end{define}

Below we define a property of matrices that gives sufficient conditions for finding a scaling of a matrix which has certain row and column sums.

\begin{define}[Property-$S$]
Let $A$ be an $m \times n$ matrix over some field. We say that $A$ satisfies {\em Property-$S$} if for every zero sub-matrix of $A$ of size $a \times b$ it holds that
\begin{equation}\label{eq-propS}
 \frac{a}{m} + \frac{b}{n} \leq 1.
\end{equation}
\end{define}

For example, a square matrix has Property-S if is has a non-zero generalized diagonal. Also, notice that this property is maintained under concatenation (say, putting two matrices with the same number of columns one under the other). The following theorem is the main tool we will use. Its proof uses ideas from convex optimization but we will only need to use it as a black box.

\begin{thm}[Matrix scaling theorem, Theorem 3 in \cite{RS89} ]\label{thm-scaling}
Let $A$ be an $m \times n$ real matrix with non-negative entries which satisfies Property-$S$. Then, for every $\eps >0$, there
exists a scaling $A'$ of $A$ such that the sum of each row of $A'$ is at most  $1+\eps$ and the sum of each column of $A'$ is at
least $m/n - \eps$. Moreover, the scaling coefficients used to obtain $A'$ are all positive real numbers.
\end{thm}

In our proof will use the following easy corollary of the above theorem that appeared in~\cite{BDWY11}. This corollary is obtained by applying the matrix scaling theorem to the matrix obtained by squaring all entries of the original matrix.

\begin{cor}[Corollary  from \cite{BDWY11}]\label{cor:scaling}
Let $A = (a_{ij})$ be an $m \times n$ complex matrix which satisfies Property-$S$. Then, for every $\eps > 0$, there exists a scaling $A'$ of $A$ such that for every $i \in [m]$ $$ \sum_{j \in [n]} |a_{ij}|^2 \leq 1+\eps$$ and for every $j \in [n]$ $$ \sum_{i \in [m]} |a_{ij}|^2 \geq m/n - \eps.$$
\end{cor}


\section{Proof of the rank bound}\label{sec:rb}
In this section we will present the proofs for Theorem~\ref{thm:improvedrankbound} and Theorem~\ref{thm:averageq}. The proof follows the same general outline as the one appearing in \cite{BDWY11}:


\paragraph{\bf Step 1 -- Scaling:}
Given the design matrix $A$, we construct a scaling $A'$ of $A$ where every column has large $\ell_2$ norm and every row has small $\ell_2$ norm. 
Since $A$ need not satisfy Property-$S$, we are not be able to apply Corollary~\ref{cor:scaling} directly. Instead we first find a matrix $B$ whose rows are chosen from 
the rows of A, with repetitions, such that no row is chosen too many times. 
If each row of $A$ occurs a maximum of $c$ times in $B$, then we call $B$ a $c$-cover of $A$.  We can then apply Corollary~\ref{cor:scaling} to get a scaling of $B$ with equal row norms and equal column norms, and then use the scaling of $B$ to derive a scaling of $A$ with the desired properties.  In~\cite{BDWY11} $B$ is taken to be a $q$-cover of $A$.

\paragraph{\bf Step 2 -- Obtaining a diagonal dominant matrix:} 
Given the scaling $A'$ of $A$  we consider the matrix $M = A'^\ast A'$. Clearly all the diagonal entries of $M$, which correspond to the squared $\ell_2$ norms of the columns of $A'$, are large. We use the design properties of $A$, as well as the properties of the scaling to show that the sum of square of the off-diagonal entries of $M$ is small. Matrices such as $M$ are called `diagonal dominant' and bounding their rank can be done in various ways (see e.g., \cite{Alo09}).  In this step, our calculation gives a tighter analysis of the bounds of the entries of $M$ and we are hence able to obtain the stronger rank bounds compared to \cite{BDWY11}.

\subsection{Covering lemmas}

The following  lemma is implicit in ~\cite{BDWY11} and shows how one can find a cover of a matrix $A$ that satisfies Property-$S$. Recall that a matrix $B$ is a $c$-cover of $A$ if each row of $B$ is a row of $A$ and each row of $A$ appears at most $c$ times in $B$.

\begin{lem}\label{lem:qcover}
Let $A$ by an $m \times n$ matrix over $\C$ that is a $(q,k,t)$ design matrix. Then there exists an $nk \times n$ matrix $B$ that is a $q$-cover of $A$, and such that $B$ satisfies Property-$S$.
\end{lem}
\begin{proof}[Proof sketch]
$B$ is constructed as follows: for each $i \in [n]$, we let $B_i$ be a $k \times n$ submatrix of $A$ which has no zeros in the $i$'th column. Let $B$ be the $nk \times n$ matrix which is composed of the concatenation all matrices $B_i, i \in [n]$. 
\end{proof}

We now prove a variant of Lemma~\ref{lem:qcover} where every row of $A$ appears at most $6nt/k$ times in $B$. In some settings $6nt/k$ might be smaller than $q$ and then this variant might give a potentially stronger rank bound (as stated in Theorem~\ref{thm:averageq}). Observe that this rank bound is independent of the parameter $q$. 

\begin{lem}\label{lem:ntkcover}
Let $A$ by an $m \times n$ matrix over $\C$ that is a $(q,k,t)$ design. Suppose $k \leq nt$, then there exists an $nk \times n$ matrix $B$ that is a $6nt/k$-cover of $A$, and that satisfies Property-$S$.
\end{lem}
\begin{proof}
We split the set of $n$ columns of $A$ into $\ell$ sets, where $\ell \leq \lceil 2nt/k \rceil$, each of size at most $k/2t$. Call these sets $S_1, S_2, \ldots, S_\ell$. For each $S_i$, we will first construct a matrix $B_i$ which is an $(|S_i| \cdot k) \times n$ matrix that will be composed of the rows of $A$, where each row appears at most $2$ times.
The matrix $B$ will be an $nk \times n$ matrix which is composed of all the matrices $B_i$, $1 \leq i \leq \ell$. In other words the set of rows of $B$ is the multi-set obtained by taking all the rows of all the $B_i$. 

The matrix $B_i$ is constructed as follows. For each column in $S_i$, there are at least $k/2$ rows such that none of the other columns in $S_i$ has a nonzero entry in that row. This is because the intersection of the support of any two columns has size at most $t$, and there are at most $k/2t$ columns in $S_i$. Thus the support of each column can intersect the union of the support of all other columns in $S_i$ in at most $k/2$ locations. For each column of $S_i$, pick some $k/2$ rows such that none of the other columns in $S_i$ has support which intersects that row, and add $2$ copies of each of those rows to the matrix $B_i$. Do this for each column in $S_i$. It follows immediately from construction that each row of $A$ appears at most two times in each $B_i$.

Since each row of $A$ appears at most two times in each $B_i$, and $B$ is composed of $\leq  \lceil 2nt/k \rceil$ such matrices $B_i$, thus  each row of $A$ appears at most $2 \lceil 2nt/k \rceil \leq 4nt/k + 2 \leq 6nt/k$ times in $B$ (using the bound $k \leq nt$).

To see that $B$ satisfies property-S, observe that $B$ can be written as the union of $k$ square $n \times n$ matrices each with nonzero entries on the main diagonal. For each of the $n\times n$ matrices, we would take $|S_i|$ rows per set $S_i$, where each row corresponds to one of the columns in $S_i$, such that none of the other columns in $S_i$ have a nonzero entry in that row.
\end{proof}

The next lemma shows the one can use a cover of $A$ to find a good scaling of $A$. This lemma is also implicit in \cite{BDWY11} and we give the proof sketch only for completeness.

\begin{lem}\label{lem:covertoscaling}
Let $A$ by an $m \times n$ matrix over $\C$, and let $B$ be an $nk \times n$ matrix that is a $c$-cover of $A$, and such that $B$ satisfies Property-$S$. Then, for every $\eps>0$, there exists a scaling $A'$ of  $A$ in which each row of $A'$ has $\ell_2$ norm at most $\sqrt{1 + \epsilon}$ and each column of $A'$ has $\ell_2$ norm at least $\sqrt{(k - \epsilon)/c}$.
\end{lem}
\begin{proof}[Proof sketch]
Fix  $\epsilon>0$ and apply Corollary~\ref{cor:scaling} on $B$ to obtain a scaling $B'$ of $B$ such that the $\ell_2$ norm of each column is at least $\sqrt{k-\epsilon}$, and the $\ell_2$ norm of each row is at most $\sqrt{1+\epsilon}$. We now use this scaling $B'$ of $B$ to obtain the scaling $A'$ of $A$.  The scaling of the columns used to get $A'$ is the same as the scaling coefficients for the columns of $B'$. We pick the scaling coefficients of the rows of $A'$ as follows: for each row $R$ that appears in $A$, we look at the occurrences $R_1, R_2, \ldots, R_i$ of the same row in $B$ and look at the scaling coefficients for those rows in $B'$. Say the coefficients are $s_1,s_2, \ldots, s_i$ $(i \leq c)$, then we take $\max\{s_1, \ldots s_i\}$ to be the scaling coefficient of row $R$. If the row $R$ does not appear in $B$ at all, then we pick the scaling coefficient to be such that the final $\ell_2$ norm of the row is $1$. One can easily verify  that $A'$ is a scaling of $A$ with each row of $A'$ having $\ell_2$ norm at most $\sqrt{1 + \epsilon}$ and each column having $\ell_2$ norm at least $\sqrt{(k - \epsilon)/c}$.
\end{proof}

\subsection{Proof of Theorem~\ref{thm:improvedrankbound} and Theorem~\ref{thm:averageq}}

Before proving the theorems we prove two more lemmas. The first lemma is the main new ingredient in our proof  enables us to get a tighter bound on the entries of the diagonal dominant matrix $M$ (see proof outline above).

\begin{lem}\label{lem:sumofsqs}
Let $A$ be an $m \times n$ matrix over $\C$. Suppose that each row of $A$ has $\ell_2$ norm $< \alpha$ and suppose that the supports of every two columns of $A$ intersect in at most $t$ locations. Let $M = A^\ast  A$. Then $$\sum_{i \neq j} |M_{ij}|^2 \leq t m \alpha^4.$$
Moreover if we know that the size of the support of every row in $A$ is at most $q$, then $$\sum_{i \neq j} |M_{ij}|^2 \leq \left(1 - \frac{1}{q}\right)t m \alpha^4.$$
\end{lem}
\begin{proof}
For $1 \leq i \leq n$, Let $C_i$ denote the $i$th column of $A$.  Then 
\begin{align*}
\sum_{i \neq j} |M_{ij}|^2 &= \sum_{i \neq j} |\langle C_i,C_j \rangle|^2 \\
&= \sum_{i \neq j} \left|\sum_{k=1}^m A_{ki}\overline A_{kj}\right|^2 \\
&\leq \sum_{i \neq j} t \sum_{k=1}^m |A_{ki}|^2|A_{kj}|^2 \\
& \leq t \sum_{k=1}^m \left(\sum_{i = 1}^n |A_{ki}|^2\right)^2 \\
&\leq t m \alpha^4.
\end{align*}

When there are at most $q$ nonzero entries per row, we have
\begin{align*}
 t \sum_{k=1}^m  \sum_{i \neq j}  |A_{ki}|^2|A_{kj}|^2 &=  
 t \sum_{k=1}^m \left(\sum_{i = 1}^n |A_{ki}|^2\right)^2 -  t \sum_{k=1}^m \left(\sum_{i = 1}^n |A_{ki}|^4\right) \\
 &\leq t \sum_{k=1}^m \left(\sum_{i = 1}^n |A_{ki}|^2\right)^2
 - t \sum_{k=1}^m \frac{1}{q} \left(\sum_{i = 1}^n |A_{ki}|^2\right)^2 \\
&= \left(1 - \frac{1}{q} \right) t  \sum_{k=1}^m\left(\sum_{i = 1}^n |A_{ki}|^2\right)^2 \\
&\leq \left(1 - \frac{1}{q} \right) t  m \alpha^4.
\end{align*}
\end{proof}

The second lemma is a variant of a folklore lemma on the rank of diagonal dominant matrices (see, e.g., \cite{Alo09}).
\begin{lem}\label{lem-diagdominant}
Let $M$ be an $n \times n$ Hermitian matrix such that for each $i \in [n]$, $M_{ii} \geq L$, where $L$ is some positive real number. Then,
$$ \rank(M) \geq \frac{n^2L^2}{nL^2 + \sum_{i \neq j}|M_{ij}|^2}. $$
\end{lem}
\begin{proof}
First, note that, w.l.o.g, we can assume $M_{ii}=L$ for all $i \in [n]$. If not, we can replace $M$ with a scaling $M'$ of $M$ defined as $M'_{ij} = \frac{L}{\sqrt{M_{ii}M_{jj}}} \cdot M_{ij}$. Since all scaling coefficients are at most $1$  we have $$ \sum_{i \neq j}|M'_{ij}|^2  \leq \sum_{i \neq j}|M_{ij}|^2$$ and both matrices $M$ and $M'$ have the same rank. To bound the rank of $M$ (assuming all diagonal elements are equal to $L$) we  denote its (real) non-zero eigenvalues by $\lambda_1, \lambda_2, \ldots, \lambda_r$, where $r = \rank(M)$. Then
\begin{align*}
n^2L^2 &= \mathsf{tr}(M)^2 = \left(\sum_{i = 1}^r \lambda_i\right)^2 \\
&\leq r \sum_{i = 1}^r \lambda_i^2 = r  \sum_{i,j = 1}^n |M_{ij}|^2 \\
&=r  \left(n  L^2 +  \sum_{i \neq j}|M_{ij}|^2 \right).
\end{align*}
Rearranging, we get the required bound on  $r$.
\end{proof}

\begin{proof}[Proof of Theorem~\ref{thm:improvedrankbound}]

Let $A$ be an $(q,k,t)$ design matrix and fix some $\eps > 0$. Using Lemmas~\ref{lem:qcover} and~\ref{lem:covertoscaling}, we obtain a scaling $A'$ of $A$ where each row of $A'$ has $\ell_2$ norm at most $\sqrt{1 + \epsilon}$ and each column has $\ell_2$ norm at least $\sqrt{(k - \epsilon)/q}$. Let $M = A'^\ast  A'$. Then $M_{ii} \geq (k- \epsilon)/q$ and, by Lemma~\ref{lem:sumofsqs}, $\sum_{i \neq j} |M_{ij}|^2 \leq \left(1 - \frac{1}{q}\right)t  m  (1 + \epsilon)^2.$ Applying Lemma~\ref{lem-diagdominant} to $M$ we get that
$$ \rank(M) \geq \frac{n^2\left(\frac{k-\epsilon}{q}\right)^2}{n  \left(\frac{k-\epsilon}{q}\right)^2 + \left(1 - \frac{1}{q}\right) t m  (1 + \epsilon)^2 }. $$ Taking $\eps$ to zero and simplifying, we get	 $$\rank(M) \geq \frac{n}{1 + \frac{q(q-1)mt}{nk^2}} \geq n - mtq(q-1)/k^2$$ where the second inequality follows from the fact that $1/(1+x) \geq 1-x$ for all $x$. Since $\rank(A) = \rank(A') \geq \rank(M)$, Theorem~\ref{thm:improvedrankbound} follows.	
\end{proof}

\begin{proof}[Proof of Theorem~\ref{thm:averageq}]
The only change in the proof of Theorem~\ref{thm:averageq} is that instead of Lemma~\ref{lem:qcover} we use Lemma~\ref{lem:ntkcover}.  By Lemmas~\ref{lem:qcover} and~\ref{lem:covertoscaling}, we get a scaling $A'$ of $A$ where each row of $A'$ has $\ell_2$ norm at most $\sqrt{1 + \epsilon}$ and each column has $\ell_2$ norm at least $\sqrt{(k - \epsilon)k/6nt}$.  Letting $M = A'^\ast  A'$ as before, we have $M_{ii} \geq (k- \epsilon)k/6nt$ and, by Lemma~\ref{lem:sumofsqs}, $\sum_{i \neq j} |M_{ij}|^2 \leq \left(1 - \frac{1}{q}\right)t  m  (1 + \epsilon)^2.$ Applying Lemma~\ref{lem-diagdominant} as before we get $$r \geq \frac{n}{1 + \frac{6mnt^3}{k^4}} \geq n - 6mn^2t^3/k^4.$$
\end{proof}

\paragraph{Tight examples for Theorem~\ref{thm:improvedrankbound}: } One might hope that the bound of $tm$ in the Theorem~\ref{thm:improvedrankbound} (where $t$ is the maximum intersection of any two columns) can be replaced by $\bar t m$, where $\bar t$ is some kind of average of intersections of all pairs of columns. 
One attempt towards showing such a statement would be by introducing another Cauchy-Schwarz after the one used in the proof of Lemma~\ref{lem:sumofsqs}. However, this does not seem to help. The resulting bound after the second Cauchy-Schwarz seems to give a bound of $\sqrt{\sum_{i \neq j} t_{ij}^2 } m$, which is worse than $t m$.  To see why {\it max} $t$ should not be replaceable with {\it average} $t$ in the final statement of the rank bound is by the example of a square $n \times n$ matrix $A$ with $n/s$ blocks of size $s\times s$ arranged along the diagonal. These blocks have all entries equal to $1$, and $0$ elsewhere. Then the rank of this matrix is $n/s$. Now, max $t$ equals $s$. If $n$ is much larger than $s$, then average $t$ is much smaller than $1$. Now take this matrix $A$ and randomly pick $n/100$ columns of $A$ and call this new matrix $A'$. Then we can still have that average $t < 1$, $k = s$ and $q = s/100$. Plugging it into our rank bound for design matrices (with max $t$ replaced by average $t$) will give us a lower bound of $\Omega(n)$ on the rank of $A'$, whereas we know that rank of $A'$ is at most $n/s$.

\section{Rank bound for square design matrices}\label{sec:sqrb}

We start by deriving an easy corollary that bounds the rank of square design matrices. 

\begin{thm}\label{thm:sq-design}
Let $A$ be an $n\times n$ matrix that is a $(q,k,t)$ design matrix. Then $\rank(A) = \Omega(k^4/nt^3)$. 
\end{thm}

\begin{proof}

Delete any $n - k^4/10nt^3$ columns of $A$ to get a new matrix $A'$. Then $A'$ is a $n \times k^4/10nt^3$ matrix that is also a $(q,k,t)$ design. Applying Theorem~\ref{thm:averageq} to $A'$, and doing the calculation, we get that $\rank(A') \geq k^4/10nt^3 -  k^4/20nt^3)$. Thus $\rank(A) = \Omega(k^4/nt^3)$. 
\end{proof}

Note that the incidence matrix of the projective plane of order $p$ is a $(p^2 + p+1) \times (p^2 + p+1)$ matrix which is a $(p,p,1)$ design. 
Thus any square matrix that has the same zero-nonzero pattern as the incidence matrix of the projective plane must have linear rank.
Such a result was known for the $0-1$ valued incidence matrix of the projective plane
matrix using the argument by Alon~\cite{Alo09}. However the result above allows us to get a bound on the rank with just the information of the zero-nonzero pattern.
In contrast, over finite fields, the rank of the projective plane incidence matrix is sub-linear in $n$.

\subsection{Monotone rigidity}

One of the motivation for proving rank bounds for matrices, using only information on their support, comes from a longstanding open problem in complexity theory, known as {\em matrix rigidity}. Informally, a matrix is {\em rigid} if one cannot reduce its rank by much by changing a small number of its entries in each column. More formally, we have:
 
\begin{define}[Matrix rigidity \cite{Val77}]\label{def-rigid}
Let $A$ be an $n \times n$ matrix over some field. We say that $A$ is $(r,s)$-rigid if $A$ cannot be written as $A = L + S$ with
\begin{enumerate}
	\item $L$ a matrix of rank at most $r$ and
	\item $S$ a matrix with at most $s$ non zeros per column.
\end{enumerate}
\end{define}

In \cite{Val77}, Valiant showed that if $A$ is $(n^{\alpha},\alpha n)$-rigid (for some constant $\alpha>0$) than a linear circuit (a circuit with fan-in 2 gates, each computing a linear combination of previously computed gates) computing the mapping $x \mapsto x^tA$ cannot have both $O(n)$ size and $O(\log(n))$ depth. Since proving such lower bounds is beyond the reach of current techniques, constructing {\em explicit} rigid matrix (over any field) has become a much sought after goal. We refer the reader to the survey \cite{Lok09} for more background on this longstanding open problem. 

Using our results on design matrices, we can construct an explicit matrix that is highly rigid, as long as one only changes its non-zero entries. 

\begin{define}[Monotone rigidity]\label{def-monrigid}
Let $A$ be an $n \times n$ matrix over some field. We say that $A$ is $(r,s)$-monotonically rigid if $A$ cannot be written as $A = L + S$ with
\begin{enumerate}
	\item $L$ a matrix of rank at most $r$,
	\item $S$ a matrix with at most $s$ non zeros per column and
	\item The support of $S$  is contained in that of $A$ (that is, $S_{ij} \neq 0$ implies $A_{ij} \neq 0$).
\end{enumerate}
\end{define}

The following is an immediate corollary of Theorem~\ref{thm:sq-design}.

\begin{cor} 
Let $A$ be an $n \times n$	matrix with non-negative real entries that is a $(q,k,t)$-design matrix with $k \geq \Omega(\sqrt n)$ and $t \leq O(1)$.  Then $A$ is $(\alpha\sqrt(n), \alpha n)$-monotonically rigid for some $\alpha > 0$. For example, one can take $A$ to be the  projective plane incidence matrix.
\end{cor}

We observe that this result can be used to derive super linear lower bounds (via Valiant's argument) for {\em monotone} circuits, which are circuits that can use linear combinations with non negative coefficients only. Such lower bounds, however, can be achieved using much simpler arguments (in fact, much stronger lower bounds of the form $\approx n^{1.5}$). We are not aware, however, of a simple way to construct monotonically rigid matrices.

\section{Proof of Theorem~\ref{thm:delta-sg} }\label{sec:apps} 

We first prove the general bound for $\delta > 0$ (later we will analyze the $\delta=1$ case). Suppose $v_1,\ldots,v_n \in \C^d$ are a $\delta$-$\SG$ configuration. Let $V$ be the $n \times d$ matrix whose $i$th row is the vector $v_i$. By shifting the points so that $v_i \neq 0$ for all $i \in [n]$ we have
 $$\affinedim\{v_1,\ldots,v_n\} = \dim\{v_1,\ldots,v_n\} -1 = \rank(V) - 1.$$  (The difference of 1 between affine and linear dimension will only matter in the $\delta=1$ case.)

Thus we want to upper-bound the rank of $V$. To do so, we will construct an $m \times n$ design-matrix $A$ so that $A  V = 0$. Then we will use the design properties of $A$ to argue that the rank of $A$ must be high, which in turn implies that the rank of $V$ must be small.  The following lemma is implicit in \cite{BDWY11} and we include its proof sketch here for completeness.

\begin{lem}\label{lem:sgtodesign}
Let $v_1, v_2, \ldots, v_n$ be a $\delta$-SG configuration. Let $V$ be the $n \times d$ matrix whose $i$'th row is the vector $v_i$. Then there exists an $m \times n$ matrix $A$ such that $A$ is a $(3,3k,6)$-design matrix with $k  = \lceil\delta(n-1)\rceil$, every row of $A$ has support of size exactly $3$, and such that $A V = 0$.
\end{lem}
\begin{proof}[Proof sketch]
A result of \cite{Hil73} on the existence of diagonal Latin squares implies that for all $r \geq 3$ there exists a set $T \subset [r]^3$ of $r^2 - r$ triples that satisfies the following properties:
\begin{enumerate}
\item Each triple $(t_1,t_2,t_3) \in T$ is of three distinct elements.
\item For each $i \in [r]$ there are exactly $3(r-1)$ triples in $T$ containing $i$ as an element.
\item For every pair $i,j \in [r]$ of distinct elements there are at most $6$ triples in $T$ which contain both $i$ and $j$ as elements.
\end{enumerate}

Let $\L$ denote the set of all special lines in the configuration. Let
$L_i$ be a subset of $\L$ containing lines passing through $v_i$. For each $\ell \in \L$ let $V_\ell$ denote the set of points
in the configuration which lie on the line $\ell$. Then $|V_\ell| \geq 3$ and we can assign to it a family of triples $T_\ell
\subset V_\ell^3$ satisfying the three properties above. We now construct the matrix $A$ by going over all special lines $\ell \in \L$ and for each triple $t = (i,j,k) \in T_\ell$ adding as a row of $A$ a vector with three non-zero coefficients in positions $i,j,k$, corresponding to the linear dependency among the collinear vectors $v_i,v_j,v_k$ so that we have $AV = 0$. We now argue that the matrix $A$ is a $(3,3k,6)$-design matrix as follows. The number of non-zeros in each row is exactly 3 by construction. For each $v_i$, there are at least $k = \lceil\delta(n-1)\rceil$ points (other than $v_i$) on special lines through $v_i$. Summing over all of these lines we get that $v_i$ appears in at least $3k$ triples and so the $i$'th column of $A$ will contain at least $3k$ non-zeros. Every distinct pair of points $v_i,v_j$ determine a unique line and so, by construction, can appear in at most 6 triples together.
\end{proof}

Given this lemma and our rank bounds from the previous section, the proof of Theorem~\ref{thm:delta-sg} follows quite easily. 
\begin{proof}[Proof of Theorem~\ref{thm:delta-sg}]
By Corollary~\ref{cor:improvedwithoutm},
$$\rank(A) \geq \frac{n}{1 + \frac{q(q-1)t}{k}}     \geq \frac{n}{1+ \frac{3\cdot 2 \cdot 6}{3\lceil \delta (n -1)\rceil}} \geq   \frac{n}{1+ \frac{12}{\delta n -1}}.$$

Now, $$\frac{n}{1+ \frac{12}{\delta n -1}} = n - \frac{12n}{\delta n + 9} > n- \frac{12}{\delta}.$$

Hence, $$\affinedim\{v_1,v_2, \ldots, v_n\} = \rank(V)-1 \leq n - \rank(A) < \frac{12}{\delta}.$$ 
\end{proof}

\subsection{The case of $\delta=1$: Kelly's Theorem}\label{sec:kelly}

We now describe how to obtain the tight bound of $2$ on the affine-dimension when $\delta=1$. In this scenario, every pair of points is on a special line. We start with the following simple claim.

\begin{claim}
Let $A$ be an $m\times n$ matrix so that $A$  is a $(q,k,t)$ design matrix and such that  the support of each row in $A$ is exactly $q$. Then  $m {q \choose 2} \leq {n \choose 2}  t$
\end{claim}
\begin{proof}
We count the number of pairs of locations in the matrix which are in the same row and are both nonzero. Counting once by rows we get that this quantity is equal to $m {q \choose 2}$. On the other hand, counting by columns (going over all pairs of columns) we get an upper bound of ${n \choose 2}t$ since two columns intersect in at most $t$ places. 
\end{proof}

Applying Lemma~\ref{lem:sgtodesign} we get an $m \times n$ matrix $A$  which is a $(3,3(n-1),6)$-design matrix. By the above claim we have that $m \leq n (n-1)$. Using Theorem~\ref{thm:improvedrankbound}  we get that $$\rank(A) \geq \frac{n}{1 + \frac{3\cdot 2 \cdot n(n-1) \cdot 6}{n (3(n-1))^2}} = \frac{n}{1 + \frac{4}{n-1}} = \frac{n(n-1)}{n+3} > n-4 .$$

Hence, $$ \rank(V) \leq n - \rank(A) < 4 .$$ 
Thus $$\dim\{v_1,v_2, \ldots, v_n\} = \rank(V) \leq 3,$$ and 
$$\affinedim\{v_1,v_2, \ldots, v_n\} \leq 2.$$

\subsection{Proof of high dimensional variant}\label{sec:hdsg}

Fix some point $v_0 \in V$. By a \emph{normalization w.r.t. $v_0$} we mean an affine transformation $N : \C^d \mapsto \C^d$ which first moves $v_0$ to zero, then picks a hyperplane $H$ s.t. no point in $V$ (after the shift) is parallel to $H$ (i.e., has inner product zero with the orthogonal vector to $H$) and finally multiplies each point (other than zero) by a constant s.t. it is in $H$. It is easy to see (see \cite{BDWY11}) that,
for such a mapping $N$, we have that  $v_0,v_1,\ldots,v_k$ are dependent
iff $N(v_1),\ldots,N(v_k)$ are dependent.


We now prove Theorem~\ref{thm-highdim-new} in two parts (corresponding to the two cases of $V$ and $V^*$). We denote by $f(\delta,k)$ the maximum $d$ such that there exists a $\delta$-$\SG_k^*$ configuration of dimensions $d$. We denote by $g(\delta,k)$ the maximum $d$ such that there exists a $\delta$-$\SG_k$ configuration of dimensions $d$. 

\begin{proof}[Proof for $\delta$-$\SG_k^*$ configurations:]
The proof is by induction on $k$. For $k=1$ we know $f(\delta,1)
\leq c / \delta$ with $c > 1$ a universal constant.
Suppose $k > 1$. We separate into two cases. The first case is when $V^*$ is a $(\delta/(2k))$-$\SG_1$ configuration and we are done using the bound on $k=1$. In the other case there is some point $v_0 \in V^*$ s.t. the size of the set of points on special lines through $v_0$ is at most $\delta/(2k)$. Let $S$ denote the set of points on special lines through $v_0$. Thus $|S| < \delta n / (2k)$. Let $N : \C^d \mapsto \C^d$ be a normalization w.r.t. $v_0$. Notice that for points $v \not\in S$ the image $N(v)$ determines $v$. Similarly, all points on some special line map to the same point via $N$.

Our goal is to show that $V' = N(V^* \setminus \{v_0\})$ is a
$((1-1/(2k))\delta)$-$SG_{k-1}^*$ configuration
(after eliminating multiplicities from $V'$).
This will complete the proof since $\dim(V^*) \leq \dim(V') + 1$.
Indeed, if this is the case we have
$$f(\delta,k) \leq \max \{ 2c (k/\delta) , f((1-1/(2k))\delta,k-1) + 1\}.$$
and, by induction, we have $f(\delta,k) \leq 4c (k/\delta)$.

Fix $v'_1,\ldots,v'_{k-1} \in V'$ to be $k-1$ independent points
(if no such tuple exists then $V'$ is trivially $1$-$\SG^*_{k-1}$ configuration).
Let $v_1,\ldots,v_{k-1} \in V^*$ be (necessarily independent) points s.t. $N(v_i) = v'_i$ for $i \in [k-1]$. Thus, there is a set $U \subset V^*$ of size at least $\delta n$ s.t.
for every $u \in U$ either $u \in \fl(v_0,v_1,\ldots,v_{k-1})$ or the $k$-flat $\fl(v_0,v_1,\ldots,v_{k-1},u)$ contains a point $w$ outside $\fl(v_0,v_1,\ldots,v_{k-1}) \cup \{ u \}$.

Let $\tilde U = U \setminus S$ so that $N$ is invertible on $\tilde U$
and $$|\tilde U| \geq |U| - |S| \geq (1-1/(2k)) \delta n.$$
Suppose $u \in \tilde U$ and let $u' = N(u)$.
If $u \in \fl(v_0,v_1,\ldots,v_{k-1})$ then $u'$ is in
$\fl(v'_1,\ldots,v'_{k-1})$.
Otherwise, $\fl(v_0,v_1,\ldots,v_{k-1},u)$ contains a point $w$ outside $\fl(v_0,v_1,\ldots,v_{k-1}) \cup \{ u \}$. Let $w' = N(w)$.
We will show that $w'$ is (a) contained in the $(k-1)$-flat $\fl(v'_1,\ldots,v'_{k-1},u')$ and (b) is outside $\fl(v'_1,\ldots,v'_{k-1}) \cup \{ u' \}$.
Property (a) follows  since $v_0$, $v_1,\ldots,v_{k-1}$, $u,w$ are dependent and so $v'_1,\ldots,v'_{k-1},u',w'$ are also dependent. To show (b) observe first that 
the points $v'_1,\ldots,v'_{k-1},u'$ are independent (since $v_0,v_1,\ldots,v_{k-1},u$ are independent) and so $u'$ is not in $\fl(v'_1,\ldots,v'_{k-1})$. We also need to show that $w' \neq u'$ but this follows from the fact that $u \neq w$
and so $w' = N(w) \neq N(u) = u'$ since $N$ is invertible on $\tilde U$
and $u \in \tilde U$.
Since $$|N(\tilde U)| = |\tilde U| \geq (1-1/(2k)) \delta n \geq (1-1/(2k)) \delta |V'|$$ the proof is complete.
\end{proof}

We now prove the second part of Theorem~\ref{thm-highdim-new}.

\begin{proof}[Proof for $\delta$-$\SG_k$ configurations]
The proof follows by induction on $k$ (the case $k=1$ is given by Theorem~\ref{thm:delta-sg}).
Suppose $k>1$. Suppose that $\dim(V) > g(\delta,k)$. We want to show that there exist $k$ independent points $v_1,\ldots,v_k$ s.t. for at least $1-\delta$ fraction of the points $w \in V$ we have
that $w$ is not in $\fl(v_1,\ldots,v_k)$ {\bf and} the flat $\fl(v_1,\ldots,v_k,w)$ is elementary (i.e., does not contain any other point).

Let $k' = g(1,k-1)$. Since we are trying to show by induction that $g(\delta,k) \leq C^k/\delta$ for some absolute constant $C$, we can pick $C$ so that
$g(\delta,k) > f(\delta,k'+1)$. Therefore,  we can find $k'+1$ independent points $v_1,\ldots,v_{k'+1}$ s.t. there is a set $U \subset V$ of size at least $(1-\delta) n$ s.t. for every $u \in U$ we have that $u$ is not in $\fl(v_1,\ldots,v_{k'+1})$ {\bf and} the $(k'+1)$-flat $\fl(v_1,\ldots,v_{k'+1},u)$ contains only one point,
namely $u$, outside $\fl(v_1,\ldots,v_{k'+1})$.

We now apply the inductive hypothesis on the set $V \cap \fl(v_1,\ldots,v_{k'+1})$ which has dimension at least $k' = g(1,k-1)$.
This gives us $k$ independent points $v'_1,\ldots,v'_{k}$ that define an elementary $(k-1)$-flat $\fl(v'_1,\ldots,v'_k)$. (Saying that $V$ is not $1$-$\SG_{k-1}$ is the same as saying that it contains an elementary $(k-1)$-flat). Joining any of the points $u \in U$ to $v'_1,\ldots,v'_k$ gives us an elementary $k$-flat and so the theorem is proved.
\end{proof}

\section{Proof of the variation on Freiman's Lemma}\label{sec:freiman}

In this section we prove Theorem~\ref{thm-freiman-new}. Let $A$ and $f: A \times A \mapsto \C^d$ be as in the statement of the theorem. The proof is divided into two claims.

\begin{claim}\label{cla-freiman1}
There exists a subset $A' \subset A$ with $|A| \geq \Omega(|A|/K)$ and $\dim(A') \leq O(K^2)$.
\end{claim} 
\begin{proof}
Let $B = A \fplus A$ so that for  every pair $(a,a') \in A \times A$ there exists some point $b \in B$ such that $f(a,a') =b$. Thus, on average, a point $b \in B$ has $|A|^2/|B|$ pairs of $A \times A$ mapping to it via the function $f$. Let $B_1$ be the set of all points in $B$ that have at least $|A|^2/10|B|$ pairs that map to it. Let $S'$ be the set of pairs of $A\times A$ that map to some element of $B_1$. Then $|S'| > |A|^2/2$. Consider $B_2 = B_1 \cup A$ and observe that $|B_2| \leq O(K|A|)$. Now, each pair in $S'$ is on a special line determined by $B_2$. Thus, by Corollary~\ref{cor:SGav2}, for $\alpha = |S'|/|B_2|^2 $, we get that 
there is a subset $B_{LD} \subset B_2$ of dimension at most $O(1/\alpha) = O(K^2)$ and size bounded by  $$|B_{LD}| \geq \Omega(\alpha |B_2|) \geq \Omega(|A|^2/|B_2|) \geq \Omega(|A|/K).$$ If $|B_{LD} \cap A| \geq |B_{LD}|/2 $, then take $A' = B_{LD} \cap A$ and the claim is proved. Otherwise, consider the set $B'_{LD} =  B_{LD} \setminus A \subset B_1$ so that $|B_{LD}'| \geq |B_{LD}|/2$.
Each point of $B'_{LD}$ has at least $|A|^2/10|B| $  pairs of $A\times A$ that map to it via the function $f$. For $a \in A$ we denote $$M(a) = \{ a' \in A \,|\, f(a,a') \in B_{LD}'\}.$$ Then the average of $|M(a)|$ (taken over all $a \in A$) is at least
$$\frac{1}{|A|}\sum_{a\in A}|M(a)| \geq \frac{|B'_{LD}|\cdot |A|^2}{10|B| \cdot |A|}= \Omega(|A|/K^2).$$ 
Call a point $a \in A$ a {\it heavy} point if $$|M(a)| \geq \frac{|B'_{LD}| \cdot |A|}{100|B|} \geq \Omega(|A|/K^2).$$

{\bf Case (1):} Some  point $a^* \in A$ has $$|M(a')| \geq \frac{K |B'_{LD}| \cdot |A|}{100|B|} = \Omega(|A|/K).$$ 
In this case, consider the set $B''_{LD} = B'_{LD} \cup \{a^*\}$. Clearly this set has dimension at most $\dim(B'_{LD}) + 1 = O(1/\alpha) + 1 = O(K^2)$. Also, the span of $B''_{LD}$ contains the set $M(a^*)$. Thus, there are $\Omega(A/K)$ points of $A$ that are contained in a set of dimension $O(K^2)$. This completes the proof of this case.

{\bf Case (2):} In this case we have that $$|M(a)| \leq \frac{K |B'_{LD}| \cdot |A|}{100|B|}$$ for all $a \in A$ and, in particular, for all heavy points.  Therefore, there must be at least $\Omega(|A|/K)$ heavy points (otherwise, the average of $|M(a)|$ would be too small). Call the set of heavy points  $H$, so that $|H| \geq \Omega(|A|/K)$. 
Pick $a_1\in H$ and consider $R_1 = \textsf{span}(B'_{LD} \cup \{a_1\}) \cap A$. Then $R_1$ contains $M(a_1)$, which has size $\Omega(|A|/K^2)$, and has dimension at most $\dim(B'_{LD}) + 1 = O(K^2)$. If $H \not\subset R_1$, we can pick some  $a_2 \in H \setminus R_1$ and define $R_2 = \textsf{span}(B'_{LD} \cup \{a_1,a_2\}) \cap A$. The dimension of $R_2$  is at most $\dim(B_{LD}') +2$, and its size is at least $|R_1| + \Omega(|A|/K^2)$, since $M(a_2) \cap M(a_1) = \emptyset$, or else  $a_2$ would be in the span of $B'_{LD} \cup \{a_1\}$. Continuing in  this manner (i.e., picking $a_3,a_4,\ldots$) for at most $K$ steps or until we run out of elements of $H$ (which has size $\Omega(|A|/K)$) we obtain a subset $A' \subset A$ of dimension at most  $O(K^2)+ K = O(K^2)$ containing at least $\Omega(A/K)$ elements. This completes the proof of the claim.
\end{proof}

\begin{claim}
We have $\dim(A) \leq O(K^2)$.
\end{claim}
\begin{proof}
Let $A' \subset A$ be a subset of size $\Omega(|A|/K)$ and dimension $O(K^2)$ given by the previous claim. Let  $T$ be a minimal set in $A \setminus A'$ for which $\textsf{span}(T \cup A')$ contains $A$. Notice that this implies that the points in $T$ are linearly independent and that  $$\dim(A' \cup T)= \dim(A') + |T|.$$

Observe that for every  $a_1 \neq a_2 \in T$, there do not exist $a_1' \neq a_2' \in A'$ and $b \in B$ such that $$f(a_1, a_1') = f(a_2, a_2') = b.$$ This is because otherwise  $\{a_1, a_2\} \subseteq \textsf{span}(A' \cup \{b\})$, which means that $\dim(\{a_1,a_2\} \cup A') \leq \dim(A') + 1$, which violates the properties of $T$. Therefore, $$|T| \cdot |A'| \leq |B|.$$ This gives $|T|  \leq O(K^2)$ and so $\dim(A) \leq \dim (A') + T = O(K^2).$
\end{proof}

\begin{cor}
If a set $A \subset \R^n$ defines at most $K|A|$ directions then $\dim(A) \leq O(K^{2})$
\end{cor}
\begin{proof}
Notice that the proof of the above theorem works also if the function $f$ is allowed to take values in projective space. Since the point at infinity on the line through $a,b$ is the direction $a-b$ we get the required consequence.
\end{proof}


\bibliographystyle{alpha}

\bibliography{optrb}

\end{document}